\documentclass{article}

\usepackage{PRIMEarxiv}
\usepackage{cite}
\usepackage{amsthm,amssymb,amsmath}
\usepackage[utf8]{inputenc} 
\usepackage[T1]{fontenc}    
\usepackage{hyperref}       
\usepackage{url}            
\usepackage{booktabs}       
\usepackage{amsfonts}       
\usepackage{nicefrac}       
\usepackage{microtype}      
\usepackage{lipsum}
\usepackage{fancyhdr}       
\usepackage{graphicx}       
\graphicspath{{media/}}     
\theoremstyle{definition}
\newtheorem{definition}{Definition}
\newtheorem{theorem}{Theorem}[section]
\newtheorem{lemma}[theorem]{Lemma}
\newtheorem{example}[theorem]{Example}

\title{Numerical Approximation of Stochastic Volterra Integral Equation Using Walsh Function
}

\author{
  Prit Pritam Paikaray, Sanghamitra Beuria, and Nigam Chandra Parida \\
  Department of Mathematics, College of Basic Science and Humanities\\
  OUAT,
		Bhubaneswar, Odisha,751003, India.\\
  \texttt{\{paikaraypritpritam@gmail.com, sbeuria108@gmail.com, ncparida@gmail.com\}} \\
}

\begin{document}
\maketitle

\begin{abstract}
This paper provides a numerical approach for solving the linear stochastic Volterra integral equation using Walsh function approximation and the corresponding operational matrix of integration. A convergence analysis and error analysis of the proposed method for stochastic Volterra integral equations with Lipschitz functions are presented. Numerous examples with available analytical solutions demonstrate that the proposed method solves linear stochastic Volterra integral equations more precisely than existing techniques. In addition, the numerical behaviour of the method for a problem with no known analytical solution is demonstrated.
\end{abstract}

\section{Introduction} 
Numerous fields, including the physical sciences, biological sciences, agricultural sciences, and financial mathematics, which includes option pricing, make extensive use of stochastic differential equations (SDE) \cite{Oksendal,Kloeden,tudor}. In these fields, stochastic Volterra integral equations (SVIE) play a crucial role. In a manner similar to other differential equations, many SDEs are practically impossible to solve, and the SVIE makes the problem even more challenging. Therefore, the numerical approximation method becomes vital when solving such problems. The approximate solution to many SVIEs can be estimated using various numerical techniques. Recently, orthogonal functions including block pulse function (BPF), Haar wavelet, Legendre polynomials, Laguerre polynomials, and Chebyshev's polynomials have been utilised to approximate the solution of SVIE \cite{Paley, Chen,Maleknejad,Mohammadi,Santanu,Saha,Basirat,Sohrabi,Cheng}.

The Walsh functions provide an orthonormal system that accepts just the values $-1$ and $1$. Because of this, a lot of mathematicians think of the Walsh system, which was developed in 1923 \cite{Walsh} and has many uses in digital technology, as an artificial orthonormal system. The fact that a computer can accurately estimate any Walsh function's current value at any given time gives it a significant edge over traditional trigonometric functions. The Walsh function was utilized by Chen and Hsiao in 1975 to resolve the variational problems \cite{Chen}. They applied a similar idea in 1979 to resolve the integral equation \cite{Hsiao}.
The technique's key property is that it transforms the problem into an algebraic system, which is then solved to yield an approximate solution to the problem. In this paper, we apply the Walsh function \cite{Walsh} to approximate the solution $x(t)$ of the following linear SVIE
 \begin{equation}    x(t)=f(t)+\int_{0}^{t}k_1(s,t)x(s)ds+\int_{0}^{t}k_2(s,t)x(s)dB(s)\label{Eq:Ram}
\end{equation}\label{SVIE}where $x(t)$, $f(t)$, $k_1(s,t)$ and $k_2(s,t)$ for $s,t\in[0,T)$, represent the stochastic processes based on the same probability space $(\Omega,F,P)$ and $x(t)$ is unknown. Here $B(t)$ is a Brownian motion \cite{Kloeden,Oksendal} and $\int_{0}^{t}k_2(s,t)x(s)dB(s)$ is the It$\hat{o}$ integral.

In most of the previous works, the evaluation is primarily based on the assumption that the derivatives $f'(t)$, $\frac{\partial^2 k_i}{\partial s \partial t}$ for $i=1, 2$, exists and bounded. Whereas in this paper, by converting BPF approximation to Walsh function approximation, we anticipate solely Lipschitz continuity of the functions $f(t), k_1(s, t)$ and $k_2 (s, t)$ which gives the same rate of convergence which is linear but it permits to consider more general form of SVIE that has to be integrated. In the last section, the approximate solution is compared with the exact solution numerically to check the validity of the method.
\section{Walsh Function and its Properties}\label{Walsh}
\begin{definition}[Rademacher Function]
	Rademacher function $r_i(t)$, $i=1,2,\hdots$, for $t\in[0,1)$ is defined  by \cite{Walsh} 
	
	$$r_i(t)=
	\Biggl\{
	\begin{aligned}
		1 & \quad \text{ $i=0$},\\
		sgn(sin(2^i\pi t)) & \quad \text{otherwise}\\
	\end{aligned}$$
	where,
	$$sgn(x)=
	\Biggl\{
	\begin{aligned}
		1 & \quad \text{ $x>0$},\\
		0 & \quad \text{ $x=0$},\\
		-1 & \quad \text{$x<0$}.
	\end{aligned}$$
\end{definition}
\begin{definition}[Walsh Function]
	The $n^{th}$ Walsh function for $n=0,1,2,\cdots,$ denoted by $w_n(t)$, $t\in[0,1)$ is defined \cite{Walsh} as 
	$$w_n(t)=(r_q(t))^{b_q}.(r_{q-1}(t))^{b_{q-1}}.(r_{q-2}(t))^{b_{q-2}}\hdots (r_{1}(t))^{b_1}$$ where $n=b_q2^{q-1}+b_{q-1}2^{q-2}+b_{q-2}2^{q-3}+\hdots +b_12^{0}$ is the binary expression of $n$. Therefore, $q$, the number of digits present in the binary expression of $n$ is calculated by $q=\big[\log_2n\big]+1$ in which $\big[\cdot\big]$ is the greatest integer less than or equal to $'\cdot'$.
\end{definition}

The first $m$ Walsh functions for $m \in \mathbb{N}$ can be written as an $m$-vector by
$$W(t)=\begin{bmatrix}
	w_0(t) & w_1(t) &w_2(t)\hdots w_{m-1}(t)
\end{bmatrix}^T.$$ The Walsh functions satisfy the following properties:

\subsection*{Orthonormality}
The set of Walsh functions is orthonormal. i.e., 
$$\int_{0}^{1}w_i(t)w_j(t)dt=\Biggl\{\begin{aligned}
	1 & \quad \text{i=j,}\\
	0 & \quad \text{otherwise}.
\end{aligned}$$
\subsection*{Completeness}
For every $f\in L^2([0,1))$ 
\begin{equation*}
	\int_{0}^{1}f^2(t)dt=\sum_{i=0}^{\infty}f_i^2\lvert\lvert w_i(t)\lvert\lvert^2
\end{equation*}
where $f_i=\int_{0}^{1}f(t)w_i(t)dt$.
\subsection*{Walsh Function Approximation}
Any real-valued function $f(t)\in L^2[0,1)$ can be approximated as 
$$f_m(t)=\sum_{i=0}^{m-1}c_iw_i(t)$$
where, $c_i=\int_{0}^{1}f(t)w_i(t)dt$.
The matrix form of the approximation is given by 
\begin{equation}
	f(t)=F^TT_WW(t) \label{Eq:2}
\end{equation}where
$ F=
\begin{bmatrix}
	f_0 &f_1 &f_2 \hdots f_{m-1}
\end{bmatrix}^T$ and 
$f_i=\int_{ih}^{(i+1)h}f(s)ds$ and $T_W$ is called the operational matrix for Walsh function.\\
One can see from \cite{Cheng} that,$$T_WT_W^T=mI \, \textrm{and} \, T_W^T=T_W$$

Similarly, $k(s,t)\in L^2([0,1)\times[0,1))$ can be approximated by
$$k_m(s,t)=\sum_{i=0}^{m-1}\sum_{j=0}^{m-1}c_{ij}w_i(s)w_j(t)$$
where, $c_{ij}=\int_{0}^{1}\int_{0}^{1}k(s,t)w_i(s)w_j(t)dtds$.\\
with the matrix form as
\begin{equation}
	k(s,t)=W^T(s)T_WKT_WW(t)=W^T(t)T_WK^TT_WW(s) \label{Eq:3}
\end{equation}
where $K=[k_{ij}]_{m\times m}, k_{ij}=\int_{ih}^{(i+1)h}\int_{jh}^{(j+1)h}k(s,t)dtds$.

In the next section, we will find a relation between block pulse function and Walsh function which later used to convert the SVIE to algebraic equation.
\section{Relationship between Walsh Function and Block Pulse Functions (BPFs)}
\begin{definition}[Block Pulse Functions]
	For a fixed positive integer $m$, an $m$-set of BPFs $\phi_i(t), t\in [0,1)$ for $i=0, 1,..., m-1$ is defined as
	$$\phi_i(t)=\biggl\{
	\begin{aligned}
		1 & \quad \text{if $\frac{i}{m}\le t < \frac{(i+1)}{m}, \quad$}\\
		0 & \quad \text{ otherwise}
	\end{aligned}
	$$
	$\phi_i$ is known as the $i$th BPF.
\end{definition}
The set of all $m$ BPFs can be written concisely as an $m$-vector,
$\Phi(t)=\begin{bmatrix}
	\phi_0(t) & \phi_1(t) &\phi_2(t)\hdots \phi_{m-1}(t)
\end{bmatrix}^T$, $t\in[0,1)$.

The BPFs are disjoint, complete, and orthogonal \cite{Hatamzadeh}.

The BPFs in vector form satisfy 
$$ \Phi(t)\Phi(t)^TX=\tilde{X}\Phi(t) \;\textrm{and}\; \Phi^T(t)A\Phi(t)=\hat{A}\Phi(t)$$
where, $X \in \mathbb{R}^{m \times 1}, \tilde{X}$ 
is the $m\times m$ diagonal matrix with $\tilde{X}(i, i)=X(i) \,\textrm{for}\, i=1, 2, 3\cdots m, A\in \mathbb{R}^{m \times m}$ and  $\hat{A}=\begin{bmatrix}
	a_{11}&	a_{22}&\hdots 	&a_{mm}
\end{bmatrix}^T$
is the $m$-vector with elements equal to the diagonal entries of $A$.
The integration of BPF vector $\Phi(t)$, $t\in[0,1)$ can be performed by \cite{Hatamzadeh}
\begin{equation}
	\int_{0}^{t}\Phi(\tau)d\tau=P\Phi(t), t\in[0,1), 
\end{equation}
	where, $P$ is called  deterministic operational matrix of integration.
	Hence, the integral of every function $f(t)\in L^2[0,1)$ can be approximated as$$\int_{0}^{t}f(s)ds=F^TP\Phi(t)$$
	Similarly, It$\hat{o}$ integral of BPF vector $\Phi(t)$, $t\in[0,1)$ can be performed by \cite{Maleknejad} as
	\begin{equation}
		\int_{0}^{t}\Phi(\tau)dB(\tau)=P_S\Phi(t), t\in[0,1)
	\end{equation}
	where, $P_S$
	is called the stochastic operational matrix of integration. Hence, the It$\hat{o}$ integral of every function $f(t)\in L^2[0,1)$ can be approximated as in \cite{Maleknejad} by $$\int_{0}^{t}f(s)dB(s)=F^TP_S\Phi(t).$$
	
	The following theorem describes a relationship between the Walsh function and the block pulse function.
	\begin{theorem}
		Let the $m$-set of Walsh function and BPF vectors are $W(t)$ and $\Phi(t)$ respectively. Then the BPF vectors $\Phi(t)$ can be used to approximate $W(t)$ as $W(t)=T_W\Phi(t)$, $m=2^k$, and $k=0,1,\hdots $, where $T_W=\big[c_{ij}\big]_{m\times m}$, $c_{ij}=w_i(\eta_j)$, for some  $\eta_j=\big(\frac{j}{m},\frac{j+1}{m}\big)$ and $i,j=0,1,2,\hdots m-1$.
	\end{theorem} 
	\begin{proof}
		Let $w_i(t)$, $i=0,1,2,\hdots m-1$, where $m=2^k$, be the $i^{th}$ element of the Walsh function vector.
		By expanding $w_i(t)$ into an $m$-term vectors of BPFs we have $w_i(t)=\sum_{j=0}^{m-1}c_{ij}\phi_j(t)=C_i^T\Phi(t)$, $i=0,1,2,\hdots m-1$, where $C_i^T$ is the $i^{th}$ row and $c_{ij}$ is the$(i,j)^{th}$ element of matrix $T_W$
		$$c_{ij}=\frac{1}{h}\int_{0}^{1}w_i(t)\phi_j(t)dt=\frac{1}{h}\int_{jh}^{(j+1)h}w_i(t)dt.$$
		By using Mean value theorem for integral we can write  $$c_{ij}=\frac{1}{h}\int_{jh}^{(j+1)h}w_i(t)dt=\frac{1}{h}\big((j+1)h-jh\big)w_i(\eta_j)=w_i(\eta_j)$$
		where $\eta_j\in \big(\frac{j}{m},\frac{j+1}{m}\big)$, $m=\frac{1}{h}$.\\
		Since $w_i(t)$ is constant in the interval $\big(\frac{j}{m},\frac{j+1}{m}\big)$, we choose $c_{ij}=w_i(\frac{2j+1}{2m})$, $i,j=0,1,2,\hdots m-1$.\\
		
		Hence $W(t)=T_W\Phi(t)$.\\
	\end{proof}	
	
	From the above theorem, it is easy to see $\Phi(t)=\frac{1}{m}T_WW(t).$\\
	With the use of above condition, we prove the following theorem: 
	\begin{lemma}[Integration of Walsh function]
		Suppose that $W(t)$ is a Walsh function vector, then the integral of $W(t)$ w.r.t. $t$ is given by \\
		$\int_{0}^{t}W(s)ds=\wedge W(t)$, where $\wedge =\frac{1}{m}T_WPT_W$ and $$P=\frac{1}{h}\begin{bmatrix}
			1 &2 &2&\hdots &2\\0 &1 &2&\hdots &2\\\vdots &\vdots &\vdots &\ddots &\vdots\\0 &0 &0&\hdots &1
		\end{bmatrix}$$
	\end{lemma}
	\begin{proof}
		Let $W(t)$ be a Walsh function vector, then the integral of $W(t)$ w.r.t. $t$ is
		$$\int_{0}^{t}W(s)ds=\int_{0}^{t}T_W\Phi(s)ds=T_W\int_{0}^{t}\Phi(s)ds$$
		$$=T_WP\Phi(t)=\frac{1}{m}\Big(T_WPT_W\Big)W(t)=\wedge W(t)$$
		where
		$\wedge=\frac{1}{m}\Big(T_WPT_W\Big)$
	\end{proof}
	Here, $\wedge$  is called as the Walsh operational matrix of integration.
	\begin{lemma}[Stochastic integration of Walsh function]
		Suppose that $W(t)$ is a Walsh function vector, then the It$\hat{o}$ integral of $W(t)$ is given by\\
		$\int_{0}^{t}W(s)dB(s)=\wedge _S W(t)$, where $\wedge_S =\frac{1}{m}T_WP_ST_W$ and $$P_S=
		\begin{bmatrix}
			B(\frac{h}{2}) &B(h)&\hdots &B(h)\\0 &B(\frac{3h}{2})-B(h)&\hdots &B(2h)-B(h)\\ \vdots &\vdots &\ddots &\vdots \\0 &0&\hdots &B(\frac{(2m-1)h}{2})-B((m-1)h)
		\end{bmatrix}.$$
	\end{lemma}
	\begin{proof}
		Let $W(t)$ be a Walsh function vector, then the It$\hat{o}$ integral of $W(t)$ is
		$$\int_{0}^{t}W(s)dB(s)=\int_{0}^{t}T_W\Phi(s)dB(s)=T_W\int_{0}^{t}\Phi(s)dB(s)$$ 
		$$=T_WP_S\Phi(t)=\frac{1}{m}\Big(T_WP_ST_W\Big)W(t)=\wedge_S W(t),$$
		where
		$\wedge_S=\frac{1}{m}\Big(T_WP_ST_W\Big)$.\\
		
	\end{proof}
	Here, $\wedge_S$  is called the Walsh operational matrix for It$\hat{o}$ integral.
	\section{Numerical Solution of  Stochastic Volterra Integral Equation}
	We consider following linear stochastic Volterra integral equation(SVIE)
	\begin{equation}\label{Eq:LSVIE}
		x(t)=f(t)+\int_{0}^{t}k_1(s,t)x(s)ds+\int_{0}^{t}k_2(s,t)x(s)dB(s)
	\end{equation}
	where $x(t)$, $f(t)$, $k_1(s,t)$ and $k_2(s,t)$ for $s,t\in[0,T)$, are the stochastic processes defined on the same probability space $(\Omega,F,P)$ and $x(t)$ is unknown. Also $B(t)$ is Brownian motion process and $\int_{0}^{t}k_2(s,t)x(s)dB(s)$ is the Ito Integral.\\
	Using equation \eqref{Eq:2} and\eqref{Eq:3} in \eqref{Eq:LSVIE} we have
	\begin{eqnarray}
		X^TT_WW(t)&=&F^TT_WW(t)+\int_{0}^{t}W^T(t)T_WK^T_1T_WW(s)W^T(s)T_WXds\nonumber\\
		&&+\int_{0}^{t}W^T(t)T_WK^T_2T_WW(s)W^T(s)T_WX dB(s)\nonumber\\
		&=&F^TT_WW(t)+W^T(t)T_WK^T_1T_W\int_{0}^{t}W(s)W^T(s)T_WXds\nonumber\\
		&&+W^T(t)T_WK^T_2T_W\int_{0}^{t}W(s)W^T(s)T_WX dB(s)\label{Laxman}	
	\end{eqnarray}
	
	Now
	\begin{eqnarray*}
		&&	\int_{0}^{t}W(s)W^T(s)T_WXds\\
		&&=\int_{0}^{t}T_W\Phi(s)\Phi^T(s)T_WT_WXds\\
		&&=mT_W\int_{0}^{t}\Phi(s)\Phi^T(s)Xds\\
		&&	=mT_W\tilde{X}\int_{0}^{t}\Phi(s)ds\\
		&&	=mT_W\tilde{X}P\frac{1}{m}T_WW(t).
	\end{eqnarray*}
	
	Hence
	\begin{equation}
		\int_{0}^{t}W(s)W^T(s)T_WXds=T_W\tilde{X}PT_WW(t)\label{int}
	\end{equation}
	Similarly,
	\begin{equation}
		\int_{0}^{t}W(s)W^T(s)T_WXdB(s)=mT_W\tilde{X}P_S\frac{1}{m}T_WW(t)=T_W\tilde{X}P_ST_WW(t)\label{intS}
	\end{equation}
	Substituting \eqref{int} and \eqref{intS} in \eqref{Laxman} and using the condition of orthonormality, we get
	\begin{eqnarray*}
		X^TT_WW(t)&=&F^TT_WW(t)+mW^T(t)T_WK^T_1\tilde{X}PT_WW(t)\\
		&&+mW^T(t)T_WK^T_2\tilde{X}P_ST_WW(t)\\
		&=&F^TT_WW(t)+W^T(t)T_WH_1T_WW(t)\\
		&&+W^T(t)T_WH_2T_WW(t)\\
		&=&F^TT_WW(t)+m\hat{H_1}^TT_WW(t)+m\hat{H_2}^TT_WW(t)\\
	\end{eqnarray*}	
	which implies that, 
	\begin{equation}
		\Big(X^T-F^T-m\hat{H_1}^T-m\hat{H_2}^T\Big)T_WW(t)=0 \label{Eq:Final1}
	\end{equation}
	where $H_1=mK^T_1\tilde{X}P$, $H_2=mK^T_2\tilde{X}P_S$.\\
	Hence 
	\begin{equation}
		\Big(X-F-m\hat{H_1}-m\hat{H_2}\Big)=[0]_{m\times 1}\label{Krishna}
	\end{equation}
	can be solved to obtain a non trivial solution of the given stochastic Volterra integral equation \eqref{Eq:LSVIE}.
	
	\section{Error Analysis} 
	In this section, we analyze the error between the approximate solution and the exact solution of the stochastic Volterra integral equation. Before we start the analysis let us define, $\|X\|_2=E(|X|^2)^\frac{1}{2}$.\\
	\begin{theorem}\label{fin}
		If $f\in L^2[0,1)$ satisfies the Lipschitz condition with Lipschitz constant $C$, then $\|e_m(t)\|_2=O(h)$, where $e_m(t)=|f(t)-\sum_{i=0}^{m-1}c_iw_i(t)|$ and  $c_i=\int_{0}^{1}f(s)w_i(s)ds$.  
	\end{theorem}
	\begin{proof}
		Let $f_m(t)=\sum_{i=0}^{m-1}c_iw_i(t)$ where $c_i=\int_{0}^{1}f(s)w_i(s)ds$.\\
		Suppose $f$ satisfies the Lipschitz condition.\\
		Now,
		$$ e_m(t)=|f(t)-f_m(t)|\le \omega(\frac{1}{2^k},f)\le  Ch.$$
		Here $\omega(\frac{1}{2^k},f)$ is called the modulus of continuity of the function $f$ \cite{Golubov}.\\
		Therefore,	$$\|e_m(t)\|_2\le Ch=O(h).$$

	\end{proof}
	
	\begin{theorem}\label{Thk}
		Suppose $k\in L^2\big([0,1)\times [0,1)\big)$ satisfies the Lipschitz condition with Lipschitz constant $L$. If  $k_m(x,y)=\sum_{i=0}^{m-1}\sum_{j=0}^{m-1}c_{ij}w_i(x)w_j(y)$, $c_{ij}=\int_{0}^{1}\int_{0}^{1}k(s,t)w_i(s)w_j(t)dtds$, then $\|e_m(x,y)\|_2=O(h)$, where $|e_m(x,y)|=|k(x,y)-k_m(x,y)|$.
	\end{theorem}
	\begin{proof}
		It is clear from \cite{Golubov} that,
		\begin{eqnarray*}
			k_m(x,y)&=&\sum_{i=0}^{m-1}\sum_{j=0}^{m-1}\bigg(\int_{0}^{1}\int_{0}^{1}k(s,t)w_i(s)w_j(t)dtds\bigg)w_i(x)w_j(y)\\
			&=&\sum_{i=0}^{m-1}\sum_{j=0}^{m-1}(\int_{0}^{1}\int_{0}^{1}k(s,t)w_i(s)w_i(x)w_j(t)w_j(y)dtds)\\
			&=&\int_{0}^{1}\int_{0}^{1}k(s,t)D_m(t\oplus y)D_m(s\oplus x)dtds\\
			&=&2^k.2^k\int_{\Delta_i^{(k)}}^{}\int_{\Delta_j^{(k)}}^{}k(s,t)dtds
		\end{eqnarray*}
		where, $D_m(t)=\sum_{i=0}^{m-1}w_i(t)$ is called the Dirichlet kernel\cite{Golubov}.\\
		Hence, $$|k_m(X)-k(X)|=2^{2k}\int_{\Delta_i^{(k)}}^{}\int_{\Delta_j^{(k)}}^{}|k(T)-k(X)|dT\\$$
		where $X=(x,y)$, $T=(s,t)$. Also note that if $k$ is uniformly Lipschitz with Lipschitz constant $L$, then
		$$|k_m(X)-k(X)|\le 2^{2k}\int_{\Delta_i^{(k)}}^{}\int_{\Delta_j^{(k)}}^{}L|T-X|dT\\$$
		
		Therefore,
		$$\|k_m(X)-k(X)\|_2\le\sqrt{2}Lh=O(h).$$
	\end{proof}
	\begin{theorem}
		Suppose $x_m(t)$ be the approximate solution of the linear SVIE \eqref{Eq:Ram}. If
		\begin{enumerate}
			\item[a)] $f\in L^2[0,1)$, $k_1(s,t)\quad \text{and} \quad k_2(s,t)\in L^2\big( [0,1)\times[0,1)\big)$  satisfies the Lipschitz condition with Lipschitz constants  $C$, $L_1$ and $L_2$ respectively,
			\item[b)]  $|x(t)|\le \sigma$, $|k_1(s,t)|\le \rho _1$ and $|k_2(s,t)|\le \rho_2$
		\end{enumerate}
		then$$ \|x(t)-x_m(t)\|_2^2=O(h^2)$$

	\end{theorem}
	\begin{proof}
		Let \eqref{Eq:Ram} be the given SVIE and $x_m(t)$ be the approximation to the solution using the Walsh function. 
		
		Then
		\begin{eqnarray*}
			x(t)-x_m(t)&=&f(t)-f_m(t)\\
			&+&\int_{0}^{t}\big(k_1(s,t)x(s)-k_{1m}(s,t)x_m(s)\big)ds\\
			&+&\int_{0}^{t}\big(k_2(s,t)x(s)-k_{2m}(s,t)x_m(s)\big)dB(s)
		\end{eqnarray*}
		that implies,
		\begin{eqnarray*}
			|x(t)-x_m(t)|&\le&| f(t)-f_m(t)|\\\nonumber
			&+&\biggl|\int_{0}^{t}\big(k_1(s,t)x(s)-k_{1m}(s,t)x_m(s)\big)ds\biggr|\\\nonumber
			&+&\biggl|\int_{0}^{t}\big(k_2(s,t)x(s)-k_{2m}(s,t)x_m(s)\big)dB(s)\biggr|.\nonumber
		\end{eqnarray*}
		We know that, $(a+b+c)^2\le 5a^2+5b^2+5c^2$
		\begin{eqnarray*}
			|x(t)-x_m(t)|^2&\le&5| f(t)-f_m(t)|^2\\\nonumber
			&+&5\biggl|\int_{0}^{t}\big(k_1(s,t)x(s)-k_{1m}(s,t)x_m(s)\big)ds\biggr|^2\\\nonumber
			&+&5\biggl|\int_{0}^{t}\big(k_2(s,t)x(s)-k_{2m}(s,t)x_m(s)\big)dB(s)\biggr|^2.\nonumber
		\end{eqnarray*}
		
		which implies that
		\begin{eqnarray}\label{eq:In}
			E\big(|x(t)-x_m(t)|^2\big)&\le&5E\biggl(| f(t)-f_m(t)|^2\biggr)
			+5I_1
			+5I_2.\nonumber
		\end{eqnarray}
		where, $$I_1=E\biggl(\biggl|\int_{0}^{t}\big(k_1(s,t)x(s)-k_{1m}(s,t)x_m(s)\big)ds\biggr|^2\biggr),$$ and
		$$I_2=E\biggl(\biggl|\int_{0}^{t}\big(k_2(s,t)x(s)-k_{2m}(s,t)x_m(s)\big)dB(s)\biggr|^2\biggr)$$
		Now for $i=1,2$, we have
		\begin{eqnarray*}
			|k_i(s,t)x(s)-k_{im}(s,t)x_m(s)|
			\le&&|k_i(s,t)||x(s)-x_m(s)|\\
			&+&|k_i(s,t)-k_{im}(s,t)||x(s)|\\
			&+&|k_i(s,t)-k_{im}(s,t)||x(s)-x_m(s)|\\ 	
		\end{eqnarray*}
		For $i=1,2$, let $|k_i(s,t)|\le\rho_i$, $|x(s)|\le\sigma$ and using Theorem \ref{Thk}, we get
		\begin{equation}\label{normkernel}
			|k_i(s,t)x(s)-k_{im}(s,t)x_m(s)| \le \sqrt{2}L_ih\sigma+ (\rho_i+\sqrt{2}L_ih)|x(t)-x_m(t)|
		\end{equation}
		which gives,
		\begin{eqnarray*}
			I_1
			&\le&E\biggl(\biggl(\int_{0}^{t}\biggl|k_1(s,t)x(s)-k_{1m}(s,t)x_m(s)\biggr|ds\biggr)^2\biggr)\\
			&\le&E\biggl(\biggl(\int_{0}^{t}\big( \sqrt{2}L_ih\sigma+ (\rho_i+\sqrt{2}L_ih)|x(t)-x_m(t)|\big)ds\biggr)^2\biggr)
		\end{eqnarray*}
		By Cauchy- Schwarz inequality, for $t>0$ and $f\in L^2[0,1)$
		$$\biggl|\int_{0}^{t}f(s)ds\biggr|^2\le t\int_{0}^{t}|f|^2ds$$
		this implies,
		\begin{eqnarray*}
			I_1
			&\le&E\biggl(2\int_{0}^{t}\biggl((\sqrt{2}L_1h\sigma)^2+(\rho_1+\sqrt{2}L_1h)^2|x(t)-x_m(t)|^2
			\biggr)ds\biggr)
		\end{eqnarray*}
		
		Therefore,
		\begin{eqnarray}\label{eq:ink}
			I_1
			&\le&2(\sqrt{2}L_1h\sigma)^2+2(\rho_1+\sqrt{2}L_1h)^2E\biggl(\int_{0}^{t}|x(t)-x_m(t)|^2ds\biggr)
		\end{eqnarray}
		Now,
		\begin{eqnarray*}
			I_2
			\le&&	E\biggl(\int_{0}^{t}\biggl|k_2(s,t)x(s)-k_{2m}(s,t)x_m(s)\biggr|^2ds\biggr)\\
			\le&&2E\biggl(\int_{0}^{t}\big((\sqrt{2}L_2h\sigma)^2+(\rho_2+\sqrt{2}L_2h)^2|x(t)-x_m(t)|^2\big)ds\biggr)
		\end{eqnarray*}
		Hence,
		\begin{eqnarray}\label{eq:inkb}
			I_2
			\le&&2(\sqrt{2}L_2h\sigma)^2+2(\rho_2+\sqrt{2}L_2h)^2E\biggl(\int_{0}^{t}|x(t)-x_m(t)|^2ds\biggr)	
		\end{eqnarray}
		
		Using Theorem \ref{fin}, equation \eqref{eq:ink} and \eqref{eq:inkb} in \eqref{eq:In}, we get
		\begin{eqnarray*}
			E\big(|x(t)-x_m(t)|^2\big)&\le&5C^2h^2\\
			&+&5\biggl(2(\sqrt{2}L_1h\sigma)^2+2(\rho_1+\sqrt{2}L_1h)^2E\biggl(\int_{0}^{t}|x(t)-x_m(t)|^2ds\biggr)\biggr)\\
			&+&5\biggl(2(\sqrt{2}L_2h\sigma)^2+2(\rho_2+\sqrt{2}L_2h)^2E\biggl(\int_{0}^{t}|x(t)-x_m(t)|^2ds\biggr)\biggr)
		\end{eqnarray*}

		\begin{eqnarray}
			E\big(|x(t)-x_m(t)|^2\big)&\le&R_1
			+R_2\int_{0}^{t}E\bigl(|x(s)-x_m(s)|^2\bigr)ds
		\end{eqnarray}
		where,
		$$R_1=5\biggl(C^2h^2+2(\sqrt{2}L_1h\sigma)^2+2(\sqrt{2}L_2h\sigma)^2\biggr)$$ and
		$$R_2=5\biggl(2(\rho_1+\sqrt{2}L_1h)^2+2(\rho_2+\sqrt{2}L_2h)^2\biggr)$$
		By using Gronwall's inequality, we have
		\begin{eqnarray}
			E\big(|x(t)-x_m(t)|^2\big)&\le&R_1\exp\biggl(\int_{0}^{t}R_2ds\biggr).
		\end{eqnarray}
		which implies that,
		\begin{equation}
			\|x(t)-x_m(t)\|_2^2	=E\big(|x(t)-x_m(t)|^2\big)\le R_1e^{R_2}=O(h^2)
		\end{equation}
	\end{proof}
	\section{Numerical Examples}\label{Numerical Examples}
In this section, we use the proposed method to solve a variety of SVIEs. The first three examples compare approximate and analytical results to demonstrate the method's convergence. Because an analytical solution is practically impossible to find, the last example illustrates approximate solutions for $m=32,64,$ and $128$ to indicate convergence.
The computations are carried out using Matlab 2013(a).

Define error $E$ as $\lVert E \rVert_\infty=\underset{1\le i\le m}{max} \lvert X_i-Y_i\rvert$, where $X_i$, $Y_i$ are the Walsh coefficient of exact solution and approximate solution respectively.
The number of iterations in the following instances is $n$, the mean of error $E$ is $\bar x_E$, and the standard deviation for error E is $s_E$.

\begin{example}\cite{Mohammadi}\label{Ex1}
	Consider the linear stochastic Volterra integral equation
	$$x(t)=1+\int_{0}^{t}cos(s)x(s)ds+\int_{0}^{t}sin(s)x(s)dB(s), s,t\in [0,0.5)$$
	with the exact solution $x(t)=\frac{1}{12}e^{\frac{-t}{4}+sin(t)+\frac{sin(2t)}{8}+\int_{0}^{t}sin(s)dB(s)}$, for $0\le t < 0.5$.
\end{example}
\begin{table}
	\caption{Mean error, standard deviation of error, and interval of confidence for mean error in Example \ref{Ex1} with m=8}
	\centering
	\begin{tabular}{l c c rr}
		\hline
		n & $\bar{x}_E$ &$s_E$ &
		\multicolumn{2}{c}{\underline{95\% interval of confidence for error mean. }}\\[1ex]
		& & &Lower &Upper\\
		\hline
		30 &0.00543042339 &0.00472214581 &0.00374062521 &0.00712022157\\
		50 &0.00626993437 &0.00442989552 &0.00504202998 &0.00749783876\\
		75 &0.00705047567 &0.00481071208 &0.00596170903 &0.00813924231\\
		100 &0.00640558992 &0.00481776079 &0.00546130880 &0.00734987103\\
		125 &0.00689936851 &0.00500025280 &0.00602278555 &0.00777595148\\
		150 &0.00686900260 &0.00580316061 &0.00594030348 &0.00779770171\\
		200 &0.00682115439 &0.00600805207 &0.00598848085 &0.00765382792\\	
		\hline
		
	\end{tabular}
\end{table}
\begin{figure}
	\centering
	\includegraphics[width=1.1\textwidth]{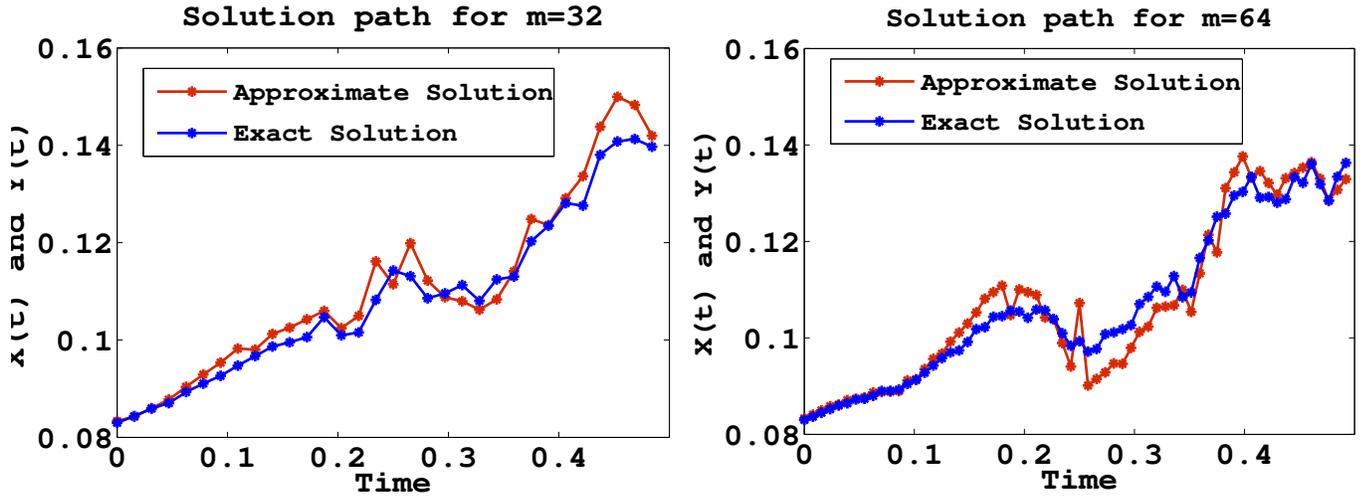}
	\caption{Example \ref{Ex1}'s approximate and exact solutions for m=32 and m=64 }
\end{figure}
\begin{figure}
	\centering
	\includegraphics[width=1\textwidth]{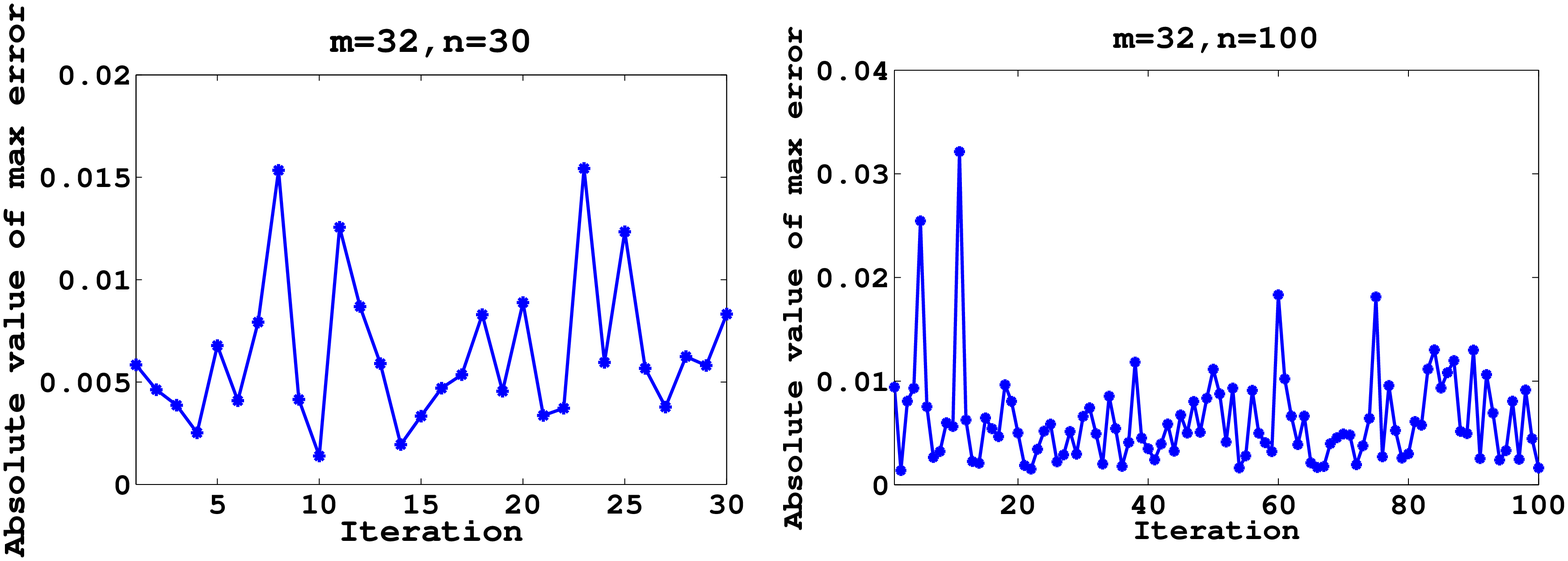}
	\caption{Example \ref{Ex1}'s error trend for m=32,n=30, and n=100}
\end{figure}

\begin{table}
	\caption{Mean error, standard deviation of error, and interval of confidence for mean error in Example \ref{Ex1} with m=32}
	\centering
	\begin{tabular}{l c c rr}
		\hline
		n & $\bar{x}_E$ &$s_E$ &
		\multicolumn{2}{c}{\underline{95\% interval of confidence for error mean. }}\\[1ex]
		& & &Lower &Upper\\
		\hline
		30 &0.00637765274 &0.00360745366 &0.00508674202 &0.00766856345\\
		50 &0.00720684095 &0.00586365605 &0.00558151841 &0.00883216349\\
		75 &0.00649984610 &0.00488908128 &0.00539334285 &0.00760634936\\
		100 &0.00625583011 &0.00474702145 &0.00532541390 &0.00718624631\\
		125 &0.00675880050 &0.00523369353 &0.00584129357 &0.00767630743\\
		150 &0.00650117417 &0.00478655986 &0.00573516505 &0.00726718328\\
		200 &0.00627666571 &0.00451326428 &0.00565115920 &0.00690217223\\	
		\hline
		
	\end{tabular}
\end{table}
\begin{example}\label{Ex3}
	Consider the linear stochastic Volterra integral equation shown below
	$$x(t)=f(t)+\int_{0}^{t}(s+t)x(s)ds+\int_{0}^{t}e^{-3(s+t)}x(s)dB(s)$$ where $s,t\in [0,1)$ in which $f(t)=t^2+sin(1+t)-cos(1+t)-2sin(t)-\frac{7t^4}{12}+\frac{1}{40}B(t)$.
\end{example}
\begin{center}
	\begin{table}
		\caption{Numerical result for m=32, m=64, and m=128 with n=50 in Example \ref{Ex3}}
		\begin{center}
			\begin{tabular}{c c|c|c}
				\hline
				$t$	&$m=2^5$  &$m=2^6$&$m=2^7$\\
				\hline
				0.1&0.2588463226&0.2786937102&0.2764612638\\
				0.2&0.2385970504&0.2482827054&0.2598350172\\
				0.3&0.2298470282&0.2362610436&0.2560260944\\
				0.4&0.2432612452&0.2547810418&0.2706006264\\
				0.5&0.3326364002&0.3482867788&0.3683207574\\
				0.6&0.3276372986&0.3354036892&0.3570142536\\
				0.7&0.3811228220&0.3973418530&0.4218934758\\
				0.8&0.4407132464&0.4640863230&0.4959036476\\
				0.9&0.5010298772&0.5335500110&0.5753660980\\
				\hline
			\end{tabular}
		\end{center}	
	\end{table}
\end{center}
\begin{figure}
	\centering
		\includegraphics[width=1.2\textwidth]{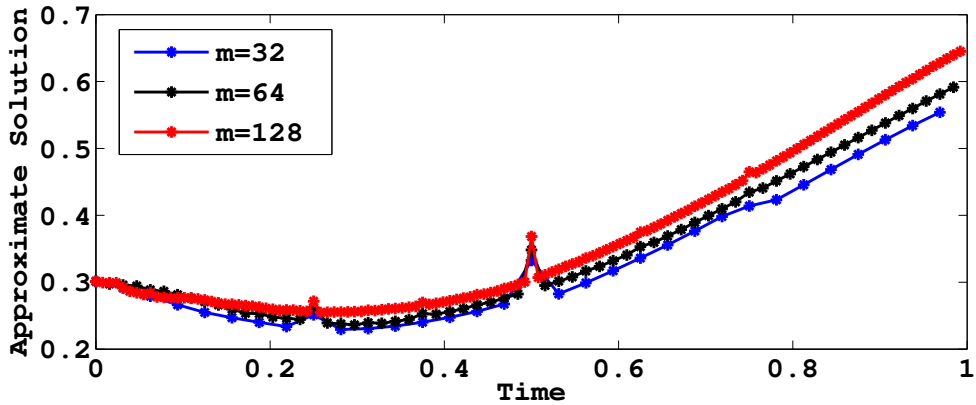}
		\caption{Example \ref{Ex3}'s approximate solution for m=32, m=64 and m=128 with 50 iterations}
	\end{figure}
\begin{example}\cite{Etheridge}\label{Stock}
	Consider the stock model with $C(t)$ as the risk-less cash bond and $S(t)$ as the single risky asset.
	$$dC(t)=\sin(t)C(t)dt, C_0=1$$
	$$S(t)=\frac{1}{10}+\int_{0}^{t} ln(1+s) S(s)ds+\int_{0}^{t}s S(s)dB(s) $$
	with the exact solution $C(t)=e^{1-cos(t)}$ and $S(t)=\frac{1}{10}e^{(1+t)ln(1+t)-t-\frac{t^3}{6}+\int_{0}^{t}sdB(s)}$.
	We will compare the exact solution of $S(t)$ with the approximate solution using our method.
\end{example}
\begin{figure}
	\centering
	\includegraphics[width=1\textwidth]{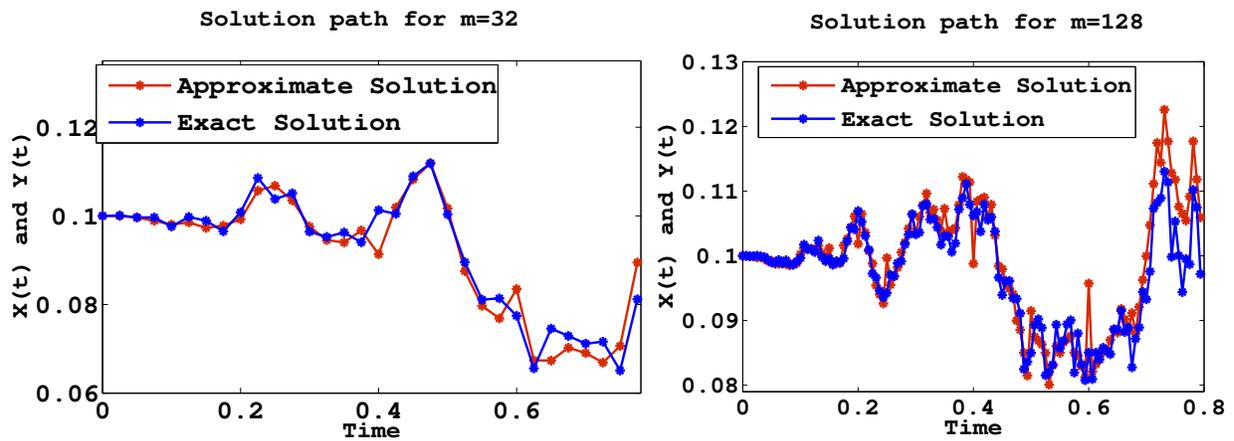}
	\caption{Stock model's approximate and exact solutions for m=32 and m=128 of Example \ref{Stock}}
\end{figure}
\begin{table}
	\caption{Mean error, standard deviation of error, and interval of confidence for mean error of Example \ref{Stock} with n=20}
	\centering
	\begin{tabular}{l c c rr}
		\hline
		$m$ & $\bar{x}_E$ &$s_E$ &
		\multicolumn{2}{c}{\underline{95\% interval of confidence for error mean. }}\\[1ex]
		& & &Lower &Upper\\
		\hline
		4 & 0.00483812406 & 0.00199063228 & 0.00396569099 & 0.00571055712\\
		8 & 0.00380827206 & 0.00251831518 & 0.00270457176 & 0.00491197235\\
		16 & 0.00432163487 & 0.00307689213 & 0.00297312744 & 0.00567014230\\
		32 & 0.00673390644 & 0.00484960640 & 0.00460847272 & 0.00885934015\\
		64 & 0.00714118035 & 0.00451414584 & 0.00516276871 & 0.00911959199\\
		128 & 0.00713451215 & 0.00627103849 & 0.00438610835 & 0.00988291594\\	
		\hline
		
	\end{tabular}
\end{table}
\begin{figure}
	\centering
	\includegraphics[width=1\textwidth]{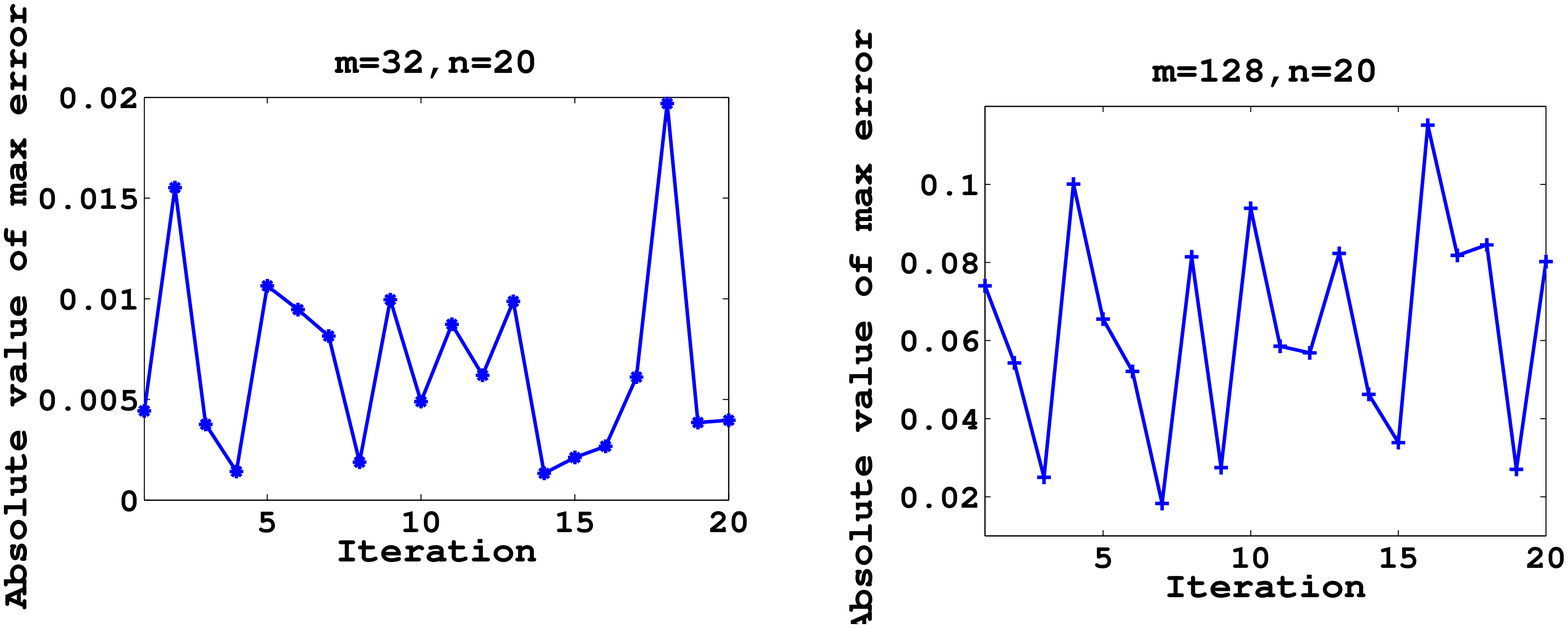}
	\caption{Example \ref{Stock}'s error trend for m=32, m=128 and n=20. }
\end{figure}	
	\section{Conclusion}
Due to the difficulty in determining the exact solution for the majority of SVIEs, numerical techniques are required to address these problems. Historically, several numerical solutions have been developed to approximate the solution of SVIEs. In addition, this article proposes a numerical method for approximating SVIE solutions. It also includes quantitative estimates for specific SVIEs. Error analysis of the methodology has been conducted to validate its dependability. As demonstrated in a number of preceding examples, numerical analysis demonstrates that the Walsh function approximation is preferable to existing methods for more precisely solving linear SVIEs. This concept could be expanded to include nonlinear SVIEs and SVIEs with singular kernels, which can be used to solve numerous physical problems.


\end{document}